\documentclass[11pt]{amsart}
\usepackage{amssymb}
\usepackage{graphicx} 
\usepackage{enumerate}
\usepackage{mathtools}
\usepackage{color,soul}
\usepackage{cite}
\usepackage{tikz}
\usetikzlibrary{matrix}
\usepackage{caption}
\usepackage{amsmath, amsthm}
\usepackage{mathtools}
 \usepackage{relsize}
 \usetikzlibrary{cd}
 \usetikzlibrary{decorations.pathreplacing}
 \usepackage{tikz,calc}
\usepackage{color}
\usepackage[margin=1.2in]{geometry}

\usepackage{multicol}

\usepackage{wrapfig}
\usepackage{cutwin}

\usepackage{lmodern}
\usepackage{enumitem}
\usepackage{lipsum}
\usepackage{stackengine}
\usepackage{appendix}
\usepackage{scrextend}

\newtheorem{thm}{Theorem}[section]
\newtheorem{thmx}{Theorem}

\newtheorem{prop}[thm]{Proposition}
\newtheorem{lem}[thm]{Lemma}
\newtheorem{cor}[thm]{Corollary}

\newtheorem{question}[thm]{Question}

\newcommand{\theoremname}{Theorem:}

\newtheorem*{conj*}{Conjecture}

  \theoremstyle{definition}
  \newtheorem{defn}[thm]{Definition}

  \newtheorem*{claim*}{Claim}

  \newtheorem*{question*}{Question}
  \newtheorem*{answer*}{Answer}
  \newtheorem*{application*}{Application}

  \theoremstyle{remark}
  \newtheorem{rmk}[thm]{Remark}
  \newtheorem*{rmk*}{Remark}

\usepackage{bbm}

\newcommand{\Odd}{\mathcal{O}_\phi}
\newcommand{\B}{\mathcal{B}}
\newcommand{\E}{\mathcal{E}_\phi}

  \DeclarePairedDelimiter\abs{\lvert}{\rvert}

  \newcommand{\from}{\colon\thinspace} 

\begin{document}

\title{{Finite image homomorphisms of the braid group and its generalizations}}

\author{Nancy Scherich}
\address{University of Toronto}
\email{n.scherich@utoronto.ca}
\urladdr{http://www.nancyscherich.com}

\author{Yvon Verberne}
\address{Mathematical Sciences Research Institute}
\email{yverberne@msri.org}


\thanks{}

\begin{abstract}
  Using totally symmetric sets, Chudnovsky, Kordek, Li, and Partin gave a superexponential lower bound on the cardinality of non-abelian finite quotients of the braid group. In this paper, we develop new techniques using \emph{multiple} totally symmetric sets to count elements in non-abelian finite quotients of the braid group. Using these techniques, we improve the lower bound found by Chudnovsky et al. We exhibit totally symmetric sets in the virtual and welded braid groups, and use our new techniques to find superexponential bounds for the finite quotients of the virtual and welded braid groups.
\end{abstract}
\maketitle

\section{Introduction}

The braid group is a versatile mathematical object which plays an important role in both topology and algebra. 
In this paper, we focus on the algebraic structure of the braid group, particularly on the size of its finite quotients. 
Many useful applications of the braid group rely on facts about finite quotients of the braid group. 
For example, the structure of Jones representations of the braid group are understood due to the fact that $B_n$ modulo the relation $\sigma_i^2=1$ is a finite group (the symmetric group, $\Sigma_n$) \cite{JONES}. 
Another example is the use of braid group representations in models of topological quantum computing. 
To have a universal quantum gate set, it is important to know the size and structure of the image of the braid group representation \cite{FLK}.

A guiding theory which motivates the work found in this paper is profinite rigity, or the idea of distinguishing groups by their finite quotients.
More specifically, one would like to understand the circumstances which allow finitely generated \textit{residually finite groups} to have isomorphic \textit{profinite completions}.
If a residually finite group $G$ is isomorphic to its profinite completion, we say that the group $G$ is \textit{profinitely rigid}.
In the context of braid groups and their generalizations, these groups are all residually finite, and the theory of profinite rigidity asks whether we can determine these groups by knowing only what their finite quotients are.
One step for studying whether a group $G$ is profinitely rigid is to determine which finite groups appear as finite quotients of the group $G$.
The work done in this paper is a step towards solving which subgroups appear as finite quotients of the braid group and its generalizations as we are providing a lower bound on the size of the non-cyclic finite quotients.
For more of an overview of recent work and progress on profinite rigidity, see \cite{Reid}.

To study finite quotients of the braid group, we consider homomorphisms $\phi \from B_n \to G$, where $G$ is a finite group. If $G$ is a cyclic group, then the quotient of $B_n$ will be a cyclic group.
A homomorphism is called \textit{cyclic} (resp. \textit{abelian}) if its image is a cyclic group (resp. an abelian group). 
One main focus of this paper is to understand the non-cyclic quotients of $B_n$. 
Work by Chudnovsky-Kordek-Li-Partin \cite{CKLP}, and more recently by Caplinger-Kordek \cite{CK}, proves a lower bound for the size of non-cyclic quotients of $B_n$.
In this paper, we provide an improved lower bound for the size of non-cylic quotients of $B_n$ by factor of $n$ to the result of Caplinger-Kordek, as found in Theorem \ref{thm:extendedbound}.

\begin{thmx}\label{thm:extendedbound}
Let $n> 5$, and let $\phi:B_n\to G$ be a non-cyclic homomorphism to a finite group, $G$. Then, 
\[\abs{\phi(B_n)} \geq  \left( \left\lfloor\frac{n}{2}\right\rfloor +1\right)(3^{\lfloor\frac{n}{2}\rfloor - 1})\left\lceil \frac{n}{2} \right\rceil!.\]
Moreover, if $p$ is the smallest integer so that $\phi(\sigma_i)^p=\phi(\sigma_j)^p$ for any $i$, $j$, then
\[ \abs{\phi(B_n)} \geq  \left(\left(\mathrm{lpf}(p)-1\right) \left\lfloor\frac{n}{2}\right\rfloor +1\right)(3^{\lfloor\frac{n}{2}\rfloor - 1})\left\lceil \frac{n}{2} \right\rceil!,\] 
where $\mathrm{lpf}(p)$ is the least integer greater than 1 that divides $p$.
\end{thmx}

A secondary motivation for Theorem \ref{thm:extendedbound} is the following conjecture.

\begin{conj*}[Margalit \cite{CKLP}]
For $n \geq 5$, $\Sigma_n$ is the smallest finite, non-cyclic quotient of $B_n$.
\end{conj*}

Theorem \ref{thm:extendedbound}, and the results by Chudnovsky-Kordek-Li-Partin and Caplinger-Kordek, attempt to rule out smaller possible non-cyclic quotients of $B_n$.
When $n=5,6,$ and $7$, Caplinger and Kordek used the classification of finite groups to conclude that a non-cyclic quotient of $B_n$ must be larger than $n!$ \cite{CK}. 
Since Theorem \ref{thm:extendedbound} gives a lower bound on the size of the image of a non-cyclic homomorphism for $n>5$, it therefore gives the tightest known lower bound for the size of a finite non-cyclic quotient of $B_n$ for $n \geq 8$.

Theorem \ref{thm:extendedbound} shows that the existence of a non-cyclic homomorphism $\phi \from B_n \to G$ requires the group $G$ to be quite large or complicated. 
To see this, recall that all finite groups embed in a large enough symmetric group, $\Sigma_k$, which implies that we can consider the target group $G$ in Theorem \ref{thm:extendedbound} to be $\Sigma_k$.
When $n\geq 6$, and $k<n$, homomorphisms $B_n\to \Sigma_k$ must be cyclic \cite{Artin}. 
Therefore, if the group $G$ embeds into a small enough symmetric group, the image of $\phi$ must be cyclic. 
However, less is known about when $k\geq n$.
One step to understand the case where $k \geq n$ was provided by Lin who showed that for $6<n<k<2n$ all transitive homomorphisms $B_n\to \Sigma_k$ are cyclic \cite{LIN}. 
Since there exist cyclic maps $B_n\to \Sigma_k$ wit $k>n$, one could ask which other types of non-cyclic homomorphisms can exist.

We prove Theorem \ref{thm:extendedbound} using totally symmetric sets inside $B_n$. 
A \textit{totally symmetric set} is a commutative set that satisfies a highly symmetric conjugation relation. 
The theory of totally symmetric sets was first introduced by Kordek and Margalit when studying homomorphisms of the commutator subgroup of $B_n$ \cite{KM}. 
More recently, totally symmetric sets were used by Caplinger-Kordek \cite{CK} and  Chudnovsky-Kordek-Li-Partin \cite{CKLP} when studying finite quotients of the braid group. 
Totally symmetric sets are useful for counting arguments since the image of a totally symmetric set $S$ under a homomorphism $\phi$ has size $\abs{\phi(S)}=\abs{S}$ or $\abs{\phi(S)}=1$.
In this paper, our approach is novel in the sense that we create counting arguments using multiple totally symmetric sets at once.

From the perspective of virtual knot theory, $B_n$ can be generalized to the virtual braid group, $vB_n$. 
One way to think of $vB_n$ is as an extension of $B_n$ by the symmetric group $\Sigma_n$, where the added permutations are the virtual crossings. 
The welded braid group, $wB_n$, is an infinite quotient of $vB_n$.
Similar to the pure braid group, the virtual and welded braid groups have ``pure" subgroups, denoted $PvB_n$ and $PwB_n$ respectively, which fix the strands of the braids pointwise. 
Inside both the virtual and welded braid groups we find totally symmetric sets.
One particularly useful type of totally symmetric set in $wB_n$ is denoted by $A_i$ in the theorems below.
Using the totally symmetric sets, $A_i$, we proved classification theorems on the size of finite images of homomorphisms for both the virtual and welded braid groups.
First we state the classification theorem for the welded braid group, $wB_n$.
We hope that this is a first step in classifying non-cyclic homomorphisms $wB_n \to G$.

\begin{thmx}\label{thm:wBn}
Let  $n> 5$, and let $\phi:wB_n\to G$ be a group homomorphism to a finite group, $G$.
One of the following must be true:
\begin{enumerate}
\item $\phi$ is abelian.
\item $\phi$ restricted to $PwB_n$ is cyclic.
\item $\abs{\phi(wB_n)}\geq 2^{n-2}(n-1)!$
\item For all $i$ and $j$, $\phi$ maps each $A_i$ to a single element with $\phi(A_i)^2\neq\phi(A_j)^2$, and 
\[\abs{\phi(wB_n)} \geq  \left( \left\lfloor\frac{n}{2}\right\rfloor +1\right)(3^{\lfloor\frac{n}{2}\rfloor - 1})\left\lceil \frac{n}{2} \right\rceil!.\]
\end{enumerate}
\end{thmx}

For the case of the virtual braid group, Bellingeri and Paris classified all homomorphisms from $vB_n \to \Sigma_k$ where $n \geq 5$, $k \geq 2$ and $n \geq k$ \cite{BP}.
However, similar to the story for $B_n$, not much is known about non-cyclic homomorphisms $vB_n\to \Sigma_k$ when $k>n$. 
Theorem \ref{Thm:vBn} is a step in the right direction towards this classification as it provides a necessary condition for the existence of a non-abelian homomorphism $vB_n \to G$.

\begin{thmx}\label{Thm:vBn}
Let $n>5$, and let $\phi \from vB_n \to G$ be a group homomorphism to a finite group, $G$. One of the following must be true:
\begin{enumerate}
    \item $\phi$ is abelian.
    \item $\phi$ factors through $wB_n$, and either
    \begin{enumerate}
        \item  $\phi$ restricted to $PwB_n$ is cyclic.
        \item $\abs{\phi(vB_n)}\geq 2^{n-2}(n-1)!$
\item For all $i$ and $j$, $\phi$ does not split $A_i$, $\phi(A_i)^2\neq\phi(A_j)^2$, and 
\[\abs{\phi(vB_n)} \geq  \left( \left\lfloor\frac{n}{2}\right\rfloor +1\right)(3^{\lfloor\frac{n}{2}\rfloor - 1})\left\lceil \frac{n}{2} \right\rceil!.\]
    \end{enumerate}
    \item $\phi$ does not factor through $wB_n$ and 
\[\abs{\phi(vB_n)} \geq  \left( \left\lfloor\frac{n}{2}\right\rfloor +1\right)(3^{\lfloor\frac{n}{2}\rfloor - 1})\left\lceil \frac{n}{2} \right\rceil!.\]
\end{enumerate}
\end{thmx}

\noindent \textit{Outline of the paper.}
Section \ref{Sec:TotallySymmetricSets} provides the background information about totally symmetric sets. 
Section \ref{Sec:ApplicationtoBn} applies these ideas to $B_n$ and gives a proof of Theorem \ref{thm:extendedbound}. 
Section \ref{Sec:TTSinVBnandwBn} provides the background about the virtual and welded braid groups, and introduces some totally symmetric sets inside of these groups. 
Section \ref{Sec:ApplicationtoWBn} contains the proofs of Theorems \ref{thm:wBn} and \ref{Thm:vBn}.\\

\noindent \textbf{Acknowledgements.}
The authors would like to thank Dror Bar-Natan and Dan Margalit for helpful conversations.
The authors would also like to thank Thomas Ng for helpful conversations which led improvement of the bound in Theorem \ref{thm:extendedbound}.
The first author was partially supported by NSERC grant RGPIN-2018-04350.
The second author was supported by the National Science Foundation under Grant No. DMS-1928930 while participating in a program hosted by the Mathematical Sciences Research Institute in Berkeley, California, during the Fall 2020 semester.
Both authors would like to thank the Georgia Institute of Technology for their hospitality during the Tech Topology Conference in December of 2019, where this project initially began.
 
 \section{Totally symmetric sets }\label{Sec:TotallySymmetricSets}
 
Kordek and Margalit introduced the theory of totally symmetric sets to give a complete classification of homomorphisms $B_n'\to B_n'$ for $n\geq 7$, where $B_n'$ is  the commutator subgroup of $B_n$ \cite{KM}. Totally symmetric sets are useful because they behave predictably under homomorphisms and group closures, as will be described in detail below.
 
 \begin{defn}\label{Def:TotallySymmetric}
 A \textit{totally symmetric set} of a group $G$ is a nonempty finite subset $\{ g_1, \ldots, g_n \}$ of $G$ which satisfies the following two relations:
 \begin{enumerate}
 \item The elements $g_i$ pairwise commute \hfill (Commutativity Condition)
 \item For every permutation $\sigma$, there exists an element $h_\sigma\in G$\\
 so that for each $i$, $h_\sigma g_{i}h_\sigma^{-1}=g_{\sigma(i)}$ \hfill (Conjugation Condition)
 \end{enumerate}
 \end{defn}
 
 \begin{rmk}
 While in our context we consider only finite totally symmetric sets, we note that totally symmetric sets need not be finite as seen in \cite{KM}.
 \end{rmk}
 
 The conjugation condition states that each permutation of $\{ g_1, \ldots, g_n \}$ can be achieved via the conjugation of an element in $G$. An important fact about totally symmetric sets is that if $f \from G \to H$ is a homomorphism and $S$ is a totally symmetric set of $G$, then $f(S)$ is a totally symmetric set of $H$.


A standard example of a group which contains totally symmetric sets is the braid group \cite{KM}. 
We begin by defining the braid group.

\begin{defn}\label{Def:BraidGroup}
The \textit{braid group} on $n$ strands, $B_n$, is the group generated by the half-twists $\sigma_1, \ldots, \sigma_{n-1}$ with the following two relations
\begin{enumerate}
    \item $\sigma_i \sigma_j = \sigma_j \sigma_i$ if $\abs{i-j}\geq2$ \hfill (Far Commutativity)
    \item $\sigma_i \sigma_{i+1} \sigma_i = \sigma_{i+1} \sigma_i \sigma_{i+1}$ if $1 \leq i \leq n-2$ \hfill (Braid Relation)
\end{enumerate}
\end{defn}
In the braid group, the subsets $S_{odd}=\{ \sigma_{2i-1} \}_{i=1}^{\lceil n/2 \rceil}$ and $S_{even}=\{ \sigma_{2i} \}_{i=1}^{\lfloor n/2 \rfloor }$ are both totally symmetric sets \cite{CKLP, KM}. 

 \subsection{The image of a totally symmetric set}\label{SubSec:ImageofTotallySymmetric}
 
 We will now discuss some of the properties which make totally symmetric sets so useful. 
 The following lemma, due to Kordek and Margalit \cite{KM}, is the crux of how totally symmetric sets are used throughout this paper. 

\begin{lem}[Kordek-Margalit]\label{lem:SplitsOrNot}  
Let  $f \from G \to H$ be a group homomorphism. Suppose that $S \subseteq G$ is a totally symmetric set of size $k$. Then $\abs{f(S)}$ is equal to either $1$ or $k$.
 \end{lem}

In this paper, we will often consider whether $\abs{\phi(S)} = \abs{S}$ or not. 
We say that $\phi$ \textit{splits} $S$ if $|\phi(S)|=|S|$.
 
 \begin{rmk}
 By Lemma \ref{lem:SplitsOrNot}, if $|S|>1$, then $\phi$ splits $S$ if and only if $|\phi(S)|>1$.
 \end{rmk}
 \begin{rmk}\label{rmk:OnlyConjugate}
The proof by Kordek and Margalit of Lemma \ref{lem:SplitsOrNot} only makes use of the conjugation condition from the definition of a totally symmetric set. 
 Therefore, Lemma \ref{lem:SplitsOrNot} holds for all sets which satisfy the conjugation condition from the definition of a totally symmetric set.
 For a set that only satisfies the conjugation condition, it makes sense to say whether $\phi$ splits the set or not.
 \end{rmk}
 
 \subsection{Totally symmetric sets with finite order elements}\label{Subsec:TotallySymmetricFiniteOrder}
 
 Let $S=\{s_i\}_{i=1}^n$ be a totally symmetric subset of a group $G$. Since all elements of $S$ are conjugate, every element of $S$ has the same order. 
 Therefore, if one element of $S$ has finite order $k \in \mathbb{N}$, every element of $S$ has order $k$. 
 In fact, if there exists $p$ so that $s_i^p=s_j^p$ for any $i,j$, then $s_i^p=s_j^p$ for \emph{all} $i,j$.  
Since the elements of a totally symmetric set commute, if a totally symmetric set consists of a finite number of elements each of finite order, then $\langle S \rangle$ is a finite group. 
The following lemma gives a lower bound of the size of this group.
A first bound was obtained by Chen, Kordek and Margalit (a proof of which can be found in \cite{CKLP}), but we will use an improvement of this bound by Caplinger and Kordek \cite{CK}(Lemma 6).
 
\begin{lem}[Caplinger-Kordek]\label{Lemma:sizeoftot}
Let $S$ be a totally symmetric set of size $k$ in a group, $G$.
Suppose further that each element of $S$ has finite order and let $p$ be the minimal integer such that $s_i^p = s_j^p$ for all $s_i, s_j \in S$. 
Then $\langle S \rangle $ is a finite group and $|\langle S\rangle|\geq p^{k-1}$.
\end{lem}
 
 Combining Lemma \ref{Lemma:sizeoftot} with the work of Chudnovsky, Kordek, Li, and Partin, one obtains a lower bound on the size of a group based on the size of a totally symmetric subset  consisting of finite order elements \cite{CKLP}.

\begin{prop}[Chudnovsky-Kordek-Li-Partin, Caplinger-Kordek]\label{Prop:SizeOfImage}
 Let $S$ be a totally symmetric set of size $k$ in a group, $G$. 
 If the elements of $S$ have finite order and $p$ is the minimal integer such that $s_i^p = s_j^p$ for all $s_i, s_j \in S$, then $|G|\geq p^{k-1}k!$.
 \end{prop}
 
 A useful restatement of Proposition \ref{Prop:SizeOfImage}  in terms of a group homomorphism is the following: Let $S$ be a totally symmetric set of a group $G$ and $\phi\from G \to H$ a group homomorphism to a finite group $H$. 
 If $\phi$ splits $S$, then $|\phi(G)|\geq p^{|S|-1}|S|!$, where $p$ is the minimal integer such that $s_i^p = s_j^p$ for all $s_i, s_j \in S$.

\section{Applications to the braid group}\label{Sec:ApplicationtoBn}

In this section, we utilize totally symmetric sets to determine a necessary condition for the existence of a non-cyclic homomorphism from the braid group into a finite group.
We begin with an overview of existing results, and then we discuss how to strengthen previous results.

\subsection{Precursory results }\label{Subsec:PreviousApplicationsBn}

Recall that two totally symmetric sets in $B_n$ are the sets $S_{odd}=\{ \sigma_{2i-1} \}_{i=1}^{\lceil n/2 \rceil}$ and $S_{even}=\{ \sigma_{2i} \}_{i=1}^{ \lfloor n/2 \rfloor  }$.
 Chudnovsky, Kordek, Li, and Partin used  these totally symmetric sets to determine a necessary condition for the existence of a non-cyclic homomorphism from the braid group into a group \cite{CKLP}.
 Recently, Caplinger and Kordek obtained a stronger necessary condition than the one found by Chudnovsky-Kordek-Li-Partin \cite{CK}.

 \begin{lem}[Caplinger-Kordek]\label{Lem:BoundonBn}
Let $G$ be a finite group and let $n \geq 5$. If the homomorphism $B_n \to G$ is non-cyclic, then
 \begin{equation}\label{Eq:BetterBoundingofImagesofBn}
   \abs{G} \geq 3^{\lfloor n/2 \rfloor -1} \left( \lfloor n/2 \rfloor \right) !.  
 \end{equation}
\end{lem}

In Section \ref{SubSec:StrengtheningBnBound}, we strengthen the lower bound found in Lemma \ref{Lem:BoundonBn}.
Before we strengthen the lower bound, we introduce the following well known fact about the braid group, which can be found in \cite{Artin}.
This lemma provides sufficient conditions for when then image of a homomorphism of the braid group is cyclic.

\begin{lem}\label{Lem:CyclicImage} For $n>4$, and let $\phi \from B_n \to G$ be a group homomorphism where $G$ is any group. 
If there exists $i, i+1\leq n-1$ so that $\phi(\sigma_i)$ commutes with $\phi(\sigma_{i+1})$ then $\phi $ is cyclic.
\end{lem}

Using Lemma \ref{Lem:CyclicImage}, we show that non-cyclic group homomorphisms split both $S_{even}$ and $S_{odd}$.

\begin{cor}\label{Cor:NoncyclicSplitsBoth}
For $n>5$, if $\phi \from B_n \to G$ is a non-cyclic group homomorphism, then $\phi$ must split both $S_{even}$ and $S_{odd}$.
\end{cor}

\begin{proof}
Suppose that $\phi$ does not split $S_{even}$. 
Then $\phi(\sigma_2)=\phi(\sigma_4)$. 
Since $\sigma_2$ commutes with $\sigma_5$, then $\phi(\sigma_4)$ commutes with $\phi(\sigma_5)$.
By Lemma \ref{Lem:CyclicImage}, $\phi$ must be cyclic, which contradicts our assumption that $\phi$ is non-cyclic.
Thus, $\phi$ must split $S_{even}$.

By a similar computation, if $\phi$ does not split $S_{odd}$, then $\phi(\sigma_4)$ commutes with $\phi(\sigma_5)$, ultimately forcing $\phi$ to be cyclic.
Therefore, $\phi$ must also split $S_{odd}$.
\end{proof}

\subsection{Proof of Theorem \ref{thm:extendedbound}}\label{SubSec:StrengtheningBnBound}
In this section, we prove Theorem \ref{thm:extendedbound}, which provides a strengthened lower bound for the smallest non-cyclic finite quotient of $B_n$.
We begin by following the proof of Chudnovsky-Kordek-Li-Partin, then further their ideas by applying Lemma \ref{Lem:CyclicImage}.\\

\noindent \textbf{Theorem A:}
\textit{Let $n> 5$, and let $\phi:B_n\to G$ be a non-cyclic homomorphism to a finite group, $G$. Then, 
\[\abs{\phi(B_n)} \geq  \left( \left\lfloor\frac{n}{2}\right\rfloor +1\right)(3^{\lfloor\frac{n}{2}\rfloor - 1})\left\lceil \frac{n}{2} \right\rceil!.\]
Moreover, if $p$ is the smallest integer so that $\phi(\sigma_i)^p=\phi(\sigma_j)^p$ for any $i$, $j$, then
\[ \abs{\phi(B_n)} \geq  \left(\left(\mathrm{lpf}(p)-1\right) \left\lfloor\frac{n}{2}\right\rfloor +1\right)(3^{\lfloor\frac{n}{2}\rfloor - 1})\left\lceil \frac{n}{2} \right\rceil!,\] 
where $\mathrm{lpf}(p)$ is the least integer greater than 1 that divides $p$.} 
\newline

\begin{proof}[Proof of Theorem \ref{thm:extendedbound}]
In this proof, denote $\Odd=\phi(S_{odd})$, $\E=\phi(S_{even})$, $s_i=\phi(\sigma_i)$, $k=\left \lceil{\frac{n}{2}}\right \rceil$, $\B = \phi(B_n)$, let $d$ be the order of the $s_i$'s and let $p$ be the smallest integer so that $\phi(\sigma_i)^p=\phi(\sigma_j)^p$ for any $i$, $j$.

First, we show that if $\phi$ does not factor through the symmetric group, then $p = d \geq 3$.
If the order of $s_i$ is equal to one, then $\phi$ is cyclic, which contradicts our assumption that $\phi$ is not cyclic.
If the order of $s_i$ is equal to two, then $p$ is also equal to two, and $\phi$ factors through the symmetric group $\Sigma_n$.
Therefore, we may assume that $d=ord(\sigma_i) \geq 3$. 
 Since we are aiming for a lower bound, we may assume that $\phi(B_n)$ is the smallest possible quotient. 
 In this case, Caplinger and Kordek prove that $p = d \geq 3$ in Lemma 7 of \cite{CK}.

Since $\phi$ is not cyclic and $n>5$, each $s_i$ is distinct by Corollary \ref{Cor:NoncyclicSplitsBoth}. 
Therefore, $\Odd$ is a totally symmetric subset of size $k$ in $\B$ since it is the injective image of a totally symmetric subset of size $k$ in $B_n$.

Notice that $\B$ acts by conjugation on the set of totally symmetric subsets of $\B$ of size $k$, and let $\Gamma= \mathrm{Fix}_{\B}(\Odd)$. 
This gives us a surjection $\psi \from \Gamma \to \Sigma_k$, where $\Sigma_k$ is the symmetric group on $k$ elements. 

Under this action by $\B$, $\Odd$ fixes $\Odd$ pointwise since the elements of a totally symmetric set pairwise commute. 
This shows that $\langle \Odd \rangle \subseteq \Gamma$ and, in fact, $\langle \Odd \rangle \subseteq \mathrm{ker}(\psi)$. 
By Lemma \ref{Lemma:sizeoftot}, we have that $\abs{\langle \Odd \rangle} \geq p^{k-1}$, and since $p \geq 3$, $\mathrm{ker}(\psi) \geq 3^{k-1}$. 
It follows that
\begin{equation}\label{Eq:InitialBnBound}
 \abs{\B} \geq \abs{\Gamma} = \abs{\Sigma_k} \cdot \abs{\mathrm{ker}(\psi)}\geq k!(\abs{\langle S_{\phi}\rangle})\geq k!3^{k-1}.  
\end{equation}

We now begin improving the bound on $\abs{\B}$.
Notice that $\E=\{s_{2i}\}_{i=1}^{\lfloor \frac{n}{2} \rfloor}$ is a second totally symmetric set which consists of the images of the remaining generators of $B_n$. 
The elements in $\E$ are currently not accounted for in Equation \ref{Eq:InitialBnBound}. 
 To include these elements in the bound of $\abs{\B}$, we consider when elements of $\langle \E \rangle$ are not in $\Gamma$. 
 The following observations lead us to find elements of $\langle \E \rangle$ that are not in $\Gamma$.
  We then count distinct cosets of $\Gamma$ in $\B$. 
 \\

\begin{addmargin}[2em]{0em}
\noindent \textit{\underline{Observation 1}:} Suppose there exists an $m \in \{ 1, \ldots, p-1 \}$ so that $s_i^m \in \Gamma$, then $s_i^m\in \ker\psi$.

By definition of $\Gamma$, $s_i^m$ acts on $\Odd$ by conjugation, fixing $\Odd$ set-wise. 
By the relations of the braid group, $s_i^m$ commutes with every element of $\Odd$ except for $s_{i\pm1}$. 
Since $\Odd$ is fixed set-wise, then either conjugation by $s_i^m$ swaps the elements $s_{i+1}$ and $s_{i-1}$, or fixes the elements pointwise.
Suppose first that conjugation by $s_i^{m}$ swaps the elements $s_{i+1}$ and $s_{i-1}$, meaning $s_i^m s_{i-1} s_i^{-m} = s_{i+1}$. 
Then
\[
s_{i+2}(s_{i+1})s_{i+2}^{-1} = s_{i+2}(s_i^ms_{i-1}s_i^{-m})s_{i+2}^{-1} = s_i^ms_{i-1}s_i^{-m}=s_{i+1},
\]
which shows that $s_{i+2}$ and $s_{i+1}$ commute. 
By Lemma \ref{Lem:CyclicImage}, $\phi$ must be cyclic. 
If $i$ is large enough so that either $i+1 > n-1$ or $i+2 > n-1$, an analogous argument shows that $s_{i-2}$ and $s_{i-1}$ commute and that $\phi$ is cyclic. 
In both cases, we have contradicted our assumption that $\phi$ is non-cyclic. 
Thus, conjugation by $s_i^m$ does not swap the elements $s_{i+1}$ and $s_{i-1}$, but rather fixes these elements pointwise. 
Therefore, for all $i$, conjugation by $s_i^m$ fixes every element of $S$ pointwise.
This implies that if $s_i^m\in \Gamma$, then $s_i^m\in \ker\psi$.\\

\noindent \textit{\underline{Observation 2}:} $s_i \not\in \Gamma$ for $i$ even.

 Suppose that $s_i \in \Gamma$. Observation 1 implies $s_i \in \ker\psi$ and $s_i$ commutes with every other $s_j$. 
 Lemma \ref{Lem:CyclicImage} implies that $\phi$ is a cyclic map, a contradiction.
Therefore, $s_i \not\in \Gamma$.\\

\noindent \textit{\underline{Observation 3}:} $s_i^m \not \in \Gamma$ for $i$ even and $m$ relatively prime to $p$.

Suppose that $s_i^m \in \Gamma$ for some $m \geq 2$ and $i$ even. 
By Observation 1, $s_i^m\in \ker\psi$, and commutes with every element of $\Odd$ and $\E$. 
Since $s_i^m$ commutes with every $s_j$, this implies that $s_i^m$ is central in the image of $\phi$. 
If there exists an integer $r$ so that  $(s_i^m)^r=s_i$, then this implies $s_i$ is also central in the image of $\phi$, which by Lemma \ref{Lem:CyclicImage}, implies that $\phi$ is cyclic, a contradiction. 
Thus, for integers $m$ that are relatively prime to $p=ord(s_i)$, the elements $s_i^m$ cannot be elements of $\ker\psi$, and hence are not elements in $\Gamma$. 
For the powers $m$ that are not relatively prime to $p$, there is no contradiction and it is possible for $s_i^m$ to be in $\ker\psi$.\\

\noindent \textit{\underline{Observation 4}:} $s_i^{m_1}\Gamma\neq s_j^{m_2}\Gamma$ when $i$ and $j$ are even, $m_1,m_2$ are relatively prime to $p$, and $m_1-m_2$ is relatively prime to $p$.

From Observation 3, $s_i^{m_1},s_j^{m_2} \notin \Gamma$ for every $i,j$ even and $m_1,m_2$ relatively prime to $p$. 
Suppose that $s_i^{m_1}\Gamma=s_j^{m_2}\Gamma$. 
This implies that $s_i^{-m_1} s_j^{m_2}\in \Gamma$.

If $i=j$, then $s_i^{-m_1} s_j^{m_2}\in \Gamma$ implies that $s_i^{m_2-m_1}\in \Gamma$. Then $m_2-m_1$ is not relatively prime to $p$, a contradiction.

If $i < j$, then $s_j$ commutes with $s_{i-1}$. Consider the action of $s_i^{-m_1} s_j^{m_2}$ on $s_{i-1}$ by conjugation: 
\[
s_j^{-m_2}s_i^{m_1}s_{i-1}s_i^{-m_1} s_j^{m_2}=s_i^{m_1}s_{i-1}s_i^{-m_1}.
\]
Since we supposed that $s_i^{-m_1} s_j^{m_2}\in\Gamma$, then conjugation by  $s_i^{-m_1} s_j^{m_2}$ fixes the set $\Odd$ setwise. 
Therefore, the above equation shows that $s_i^{m_1}s_{i-1}s_i^{-m_1}\in \Odd$.
The exponent $m_1$ was chosen so that $s_i^{m_1}\not\in \Gamma$, which means that $\Odd$ is not closed under conjugation by $s_i^{m_1}$. As described in Observation 1, $s_i^{m_1}s_{i-1}s_i^{-m_1}$ is not an element of $\Odd$ when $\phi$ is non-cyclic, a contradiction. Hence $s_i^{-m_1} s_j^{m_2} \not \in \Gamma$.\\

\noindent \textit{\underline{Observation 5}:} Counting the cosets of $\Gamma$.

Let $\mathrm{lpf}(p)$ be the least integer greater than 1 that divides $p$.
Notice that the set $\{ 2, \ldots, \mathrm{lpf}(p)-1, \mathrm{lpf}(p)+1 \}$ is a set of $\mathrm{lpf}(p) - 1$ integers which are each relatively prime to $p$, and pairwise their differences are relatively prime to $p$.
Together with Observation 4, this shows that there are $\mathrm{lpf}(p)-1$ distinct non-intersecting cosets of $\Gamma$ for each $s_i$ with $i$ even.
Since there are $|\E|$ distinct elements $s_i$ with $i$ even, we get $(\mathrm{lpf}(p)-1)|\E|$ distinct cosets of $\Gamma$.
By counting the distinct cosets of $\Gamma$, we obtain a lower bound for the complement of $\Gamma$ in the image of $\phi$.
\begin{equation}\label{eq:GammaBoundBn}
    \abs{\phi(B_n)-\Gamma}\geq (\mathrm{lpf}(p)-1)\abs{\E}\cdot\abs{\Gamma}
\end{equation}
\end{addmargin}

\noindent Combining all of these observations, we arrive at the following lower bound for $\abs{B_n}$,
\begin{equation}\label{Eq:BoundforBn}
\begin{aligned}
\abs{\phi(B_n)}\geq \abs{\Gamma}+((\mathrm{lpf}(p)-1)\abs{\E})\abs{\Gamma}&=((\mathrm{lpf}(p)-1)\abs{\E}+1)\cdot\abs{\Gamma}\\
&\geq((\mathrm{lpf}(p)-1)\abs{\E}+1)(\abs{\langle \Odd\rangle})k! .
\end{aligned}
\end{equation}
 By substituting the values of $k$, $\abs{\langle \Odd \rangle}$, and $\abs{\E}$, we arrive at our final result:
\begin{equation}
   \abs{\phi(B_n)} \geq  \left(\left(\mathrm{lpf}(p)-1\right) \left\lfloor\frac{n}{2}\right\rfloor +1\right)(3^{\lfloor\frac{n}{2}\rfloor - 1})\left\lceil \frac{n}{2} \right\rceil!.
\end{equation}
For all even values of $p$, we have that $\mathrm{lpf}(p)-1 = 1$, which gives the minimal bound
\begin{equation}
  \abs{\phi(B_n)} \geq  \left( \left\lfloor\frac{n}{2}\right\rfloor +1\right)(3^{\lfloor\frac{n}{2}\rfloor - 1})\left\lceil \frac{n}{2} \right\rceil!.
\end{equation}
\end{proof}

\section{Totally symmetric sets in the Virtual and Welded Braid Groups }\label{Sec:TTSinVBnandwBn}

In this section, we introduce two generalizations of the braid group, namely, the virtual braid group and the welded braid group.
For each group, we give examples of totally symmetric sets, as well as provide the important lemmas we require to prove our main results, Theorems \ref{thm:wBn} and \ref{Thm:vBn}.

\subsection{The virtual braid group}\label{SubSection:VirtualBraidGroup}

Let $vB_n$ denote the \textit{virtual braid group} on $n$ strands.
This group has generators $\sigma_1, \ldots ,\sigma_{n-1}$ and $\tau_1, \ldots, \tau_{n-1}$.
The generators $\sigma_1, \ldots ,\sigma_{n-1}$ satisfy the classical braid group relations, and the generators $\tau_1, \ldots, \tau_{n-1}$ generate the symmetric group. 
There are also some mixing relations.
We list all relations in the virtual braid group, below:
\begin{enumerate}
    \item $\sigma_i \sigma_j = \sigma_j \sigma_i$ for $\abs{i-j} > 1$ \hfill (Far Commutativity)
    \item $\sigma_i \sigma_{i+1} \sigma_i = \sigma_{i+1} \sigma_i \sigma_{i+1}$ for $1 \leq i \leq n-2$ \hfill (Braid Relation)
    \item $\tau_i^2 = 1$ for $1 \leq i \leq n-1$ \hfill ($\tau$ is a Transposition)
    \item $\tau_i \tau_j = \tau_j \tau_i$ for $\abs{i-j} > 1$ \hfill ($\tau$ Far Commutativity)
    \item $\tau_i \tau_{i+1} \tau_i = \tau_{i+1} \tau_i \tau_{i+1}$ for $1 \leq i \leq n-2$ \hfill ($\tau$ Braid Relation)
    \item $\sigma_i \tau_j = \tau_j \sigma_i$ for $\abs{i-j} > 1$ \hfill (Mixed Far Commutativity)
    \item $\tau_{i+1} \sigma_i \tau_{i+1} = \tau_i \sigma_{i+1} \tau_i$ for $1 \leq i \leq n-2$ \hfill (Mixed Braid Relation)
\end{enumerate}
We note that these relations encode the \textit{virtual (or extended) Reidemeister moves}.

From this point of view, $vB_n$ is the free product of the braid group and the symmetric group modulo relations (6) and (7), $vB_n=B_n*\Sigma_n\slash _{(6),(7)}$. 
This presentation is nice in the sense that you can ``see" the braid group as a subgroup of the virtual braid group. 
The canonical embeddings of $B_n$ and $\Sigma_n$ in $vB_n$ are $B_n=\langle \sigma_1, \cdots ,\sigma_{n-1}\rangle$ and $\Sigma_n=\langle \tau_1, \cdots ,\tau_{n-1}\rangle$.

Another presentation of $vB_n$ highlights a key difference between the virtual braids and the non-virtual braids. 
The pure virtual braid group, $PvB_n$, is a subgroup of $vB_n$ which is the kernel of the projection $vB_n\to \Sigma_n$ by sending $\sigma_i\mapsto \tau_i$ and $\tau_i\mapsto \tau_i$. 
Unlike the classical braid group, $vB_n$ splits as a semidirect product, $vB_n\cong PvB_n\rtimes \Sigma_n$ \cite{Bard}.

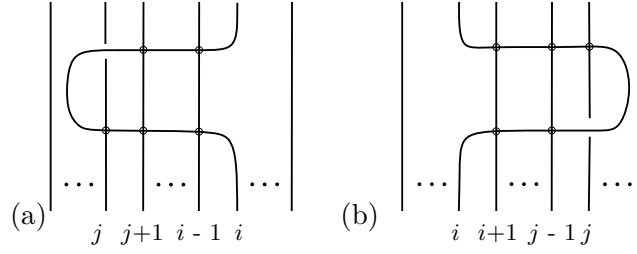
\begin{figure}[h]
\begin{picture}(100,75)
\put(-53,5){\includegraphics[width=225\unitlength]{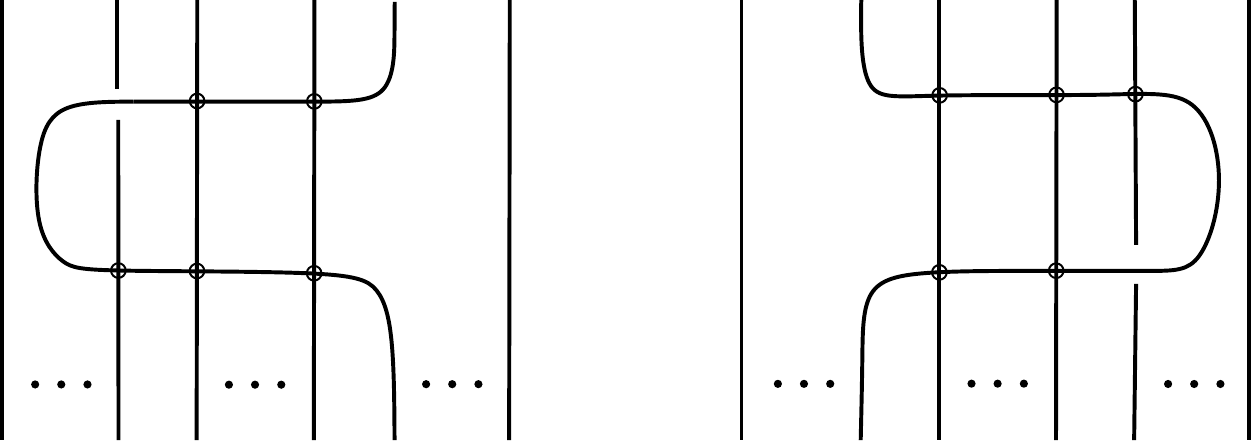}}
\put(-37,-6){\footnotesize{$j$}}
\put(-26,-6){\footnotesize{$j\text{+}1$}}
\put(-5,-6){\footnotesize{$i\text{ - }1$}}
\put(17,-6){\footnotesize{$i$}}
\put(99,-6){\footnotesize{$i$}}
\put(109,-6){\footnotesize{$i\text{+}1$}}
\put(128,-6){\footnotesize{$j\text{ - }1$}}
\put(148,-6){\footnotesize{$j$}}
\put(-68,0){(a)}
\put(57,0){(b)}
\end{picture}
\caption{(a) The element $\sigma_{i,j}$ when $j<i$. (b) The element $\sigma_{i,j}$ when $i<j$.}
\label{Fig:labellingsphere} 
\end{figure}

There are several important elements in $vB_n$, denoted $\sigma_{i,j}$, which are of the form \[
 \sigma_{i,j} = \tau_i \tau_{i+1} \ldots \tau_{j-2} \tau_{j-1} \sigma_{j-1} \tau_{j-2}  \ldots \tau_{i+1} \tau_i
 \]
 when $i<j$ and 
 \[
\sigma_{i,j} = \tau_{i-1} \tau_{i-2} \ldots \tau_{j-2} \tau_{j-1} \sigma_{j} \tau_{j} \tau_{j-1} \ldots \tau_{i-1}
 \]
 
 \noindent when $j<i$. 
 These generators are depicted in Figure \ref{Fig:labellingsphere}.
 One useful presentation for $PvB_n$ is generated by the $\sigma_{i,j}$ elements and the following two relations:
 
 \hangindent=0.7cm
\hangafter=0
 \noindent \textit{Commutivity Relation:} $\sigma_{i,j}\sigma_{k,l}=\sigma_{k,l}\sigma_{i,j}$ where $\abs{\{ i,j,k,l \}} = 4$
 
 \hangindent=0.7cm
\hangafter=0
 \noindent \textit{Braid Relation:} $\sigma_{i,j}\sigma_{i,k}\sigma_{j,k} = \sigma_{j,k}\sigma_{i,k}\sigma_{i,j}$ where $\abs{\{ i,j,k \}} = 3$

 There are many ways to embed $PvB_n$ into $vB_n$, where the presentation above is called the \textit{canonical embedding}.
 Subsection \ref{Subsec:WeldedBraids} will give more details for the generators in this embedding.

\subsubsection{Totally symmetric sets in the virtual braid group}\label{Subsec:TotSymvBn}
Since $B_n$ is a subgroup of $vB_n$, the sets $S_{odd} = \{ \sigma_{2i-1} \}_{i=1}^{\lceil n/2 \rceil} $ and $S_{even} = \{ \sigma_{2i} \}_{i=1}^{\lfloor n/2 \rfloor }$ are also totally symmetric subsets of $vB_n$. 
Additionally the sets $T_{odd} = \{ \tau_{2i-1} \}_{i=1}^{\lceil n/2 \rceil}$ and $T_{even} = \{ \tau_{2i} \}_{i=1}^{\lfloor n/2 \rfloor}$ are totally symmetric subsets of $vB_n$. 
A fun way to see why $T_{odd}$ and $T_{even}$ are totally symmetric is they are the homomorphic image of $S_{even}$ and $S_{odd}$ under the canonical projection map from $B_n\to \Sigma_n$.

The sets $\{\tau_i\sigma_i\}_{even}$ and $\{\tau_i\sigma_i\}_{odd}$ are totally symmetric sets in $vB_n$. 
They commute by a combination of relations (1),(5), and (7). 
The conjugation condition holds since you can swap $\tau_i\sigma_i$ with $\tau_{i+2}\sigma_{i+2}$ by conjugation under $\tau_{i+1}\tau_{i+2}\tau_{i}\tau_{i+1}$, which leaves all other elements of the set fixed.

\subsection{ The welded braid group }\label{Subsec:WeldedBraids} 
The welded braid group, $wB_n$, is a quotient of $vB_n$ by the \textit{Over Crossings Commute} relation, or ``OC" relation, defined as $\tau_i\sigma_{i+1}\sigma_{i}=\sigma_{i+1}\sigma_i\tau_{i+1}$ \cite{BS}.

Recall from Section \ref{Subsec:TotSymvBn}, a presentation for $PvB_n$ is generated by the elements denoted $\sigma_{i,j}$.  
These elements also generate the pure welded braid group, $PwB_n$, under the canonical embedding of $PwB_n$ into $wB_n$. 
Analogous to the virtual braid group, $wB_n$ is a semideirect product of the pure welded braid group and the symmetric group, $wB_n=PwB_n\rtimes \Sigma_n$.
Through communication with Dror Bar-Natan, we learned that the OC relation implies that $\sigma_{i,k}\sigma_{i,j}=\sigma_{i,j}\sigma_{i,k}$, and a proof of this fact can be found in \cite{SV}.
The OC relation allows us to find totally symmetric sets in $wB_n$ consisting of elements of the form $\sigma_{i,j}$.

\subsubsection{Totally symmetric sets in $wB_n$}\label{Subsec:TotallySymmetricwBn}

All of the totally symmetric sets in $vB_n$ are also totally symmetric in $wB_n$. 
Due to the OC relation, $wB_n$ has additional totally symmetric sets coming from subsets of the $\sigma_{i,j}$ elements.

If $i<j$, we call $\sigma_{i,j}$ a \textit{right generator},  and is shown in Figure \ref{Fig:labellingsphere} (b).
We denote the set of right generators with fixed $i$ as $R_i=\{\sigma_{i,j}\}_{j>i}^n$. 
If $i<j$, we call $\sigma_{i,j}$ a \textit{left generator}, and is  shown in Figure \ref{Fig:labellingsphere} (a).
For a fixed $i$, the set of left generators is denoted by $L_i=\{\sigma_{i,j}\}_{i>j}^n$.  
Let $A_i= L_i \cup R_i$ be the set of all elements of the form $\sigma_{i,j}$ which have the same first index.
The sets $A_i$, $R_i$ and $L_i$ are totally symmetric sets in $wB_n$.

\begin{lem}\label{lem:totSetsInwBn}
For each integer $1\leq i\leq n$, 
\begin{enumerate}
    \item $A_i$ is a totally symmetric set in $wB_n$ of size $n-1$.
    \item $R_i$ is a totally symmetric set in $wB_n$ of size $n-i$.
    \item $L_i$ is a totally symmetric set in $wB_n$ of size $i-1$.
\end{enumerate}
\end{lem}

\begin{proof}
Fix $i$. 
By definition, $\abs{R_i}=n-i$, $\abs{L_i}=i-1$, and $\abs{A_i}=\abs{R_i}+\abs{L_i}=n-1$.
Since $R_i$ and $L_i$ are subsets of $A_i$, it suffices to show that $A_i$ is a totally symmetric set in $wB_n$.
The elements in $A_i$ all have the same first index $i$, and commute by the OC relation. 
Therefore, we need only to show that every permutation of the elements in $A_i$ can be achieved via conjugation by an element in $wB_n$.

From the semidirect product decomposition $wB_n=PwB_n\rtimes \Sigma_n$,  $\Sigma_n$ acts on $PwB_n$ by conjugation. This action permutes the indices of the braid generators, $\sigma_j$, and in turn, permutes the pure welded braid generators, $\sigma_{i,j}$.

For $j\neq i\pm 1$, conjugation by $\tau_j$ transposes $\sigma_{i,j}$ and $\sigma_{i,j+1}$ while leaving every other $\sigma_{i,k}$ fixed. 
Conjugation by $f = \tau_{i-1}\tau_{i}\tau_{i-1}$ transposes $\sigma_{i,i-1}$ and $\sigma_{i,i+1}$, but fixes every other element in $A_i$. 
To see this, notice:
\begin{equation*}
    \begin{aligned}
    f \sigma_{i,i-1} f^{-1} &= (\tau_{i-1}\tau_{i}\tau_{i-1})(\sigma_{i-1}\tau_{i-1})(\tau_{i-1}\tau_{i}\tau_{i-1})\\
    &= \tau_{i-1}\tau_{i}\tau_{i-1} (\tau_i \tau_i) \sigma_{i-1} \tau_{i}\tau_{i-1}&\text{insert $\tau_{i}\tau_i$}\\
    &= \tau_{i-1}\tau_{i}\tau_{i-1} \tau_i (\tau_{i-1} \sigma_{i} \tau_{i-1}
    )\tau_{i-1}&\text{mixed braid relation}\\
    &= \tau_{i-1}\tau_{i}(\tau_{i} \tau_{i-1} \tau_{i}) \sigma_{i} &\text{ $\tau$-braid relation, cancel $\tau_{i-1}\tau_{i-1}$}\\
    &= \tau_{i} \sigma_{i} = \sigma_{i,i+1}. &\text{ cancel $\tau_{i}\tau_{i}$ and $\tau_{i-1}\tau_{i-1}$}
    \end{aligned}
\end{equation*}

\end{proof}

The OC relation is required for the sets $R_i$, $L_i$ and $A_i$ to satisfy the commutation condition. 
These sets are \textit{not} totally symmetric in $vB_n$, but do satisfy the conjugation condition in $vB_n$.

\subsubsection{Important lemmas}

There are many ways to embed $B_n$, $\Sigma_n$ and $PvB_n$ as subgroups inside $vB_n$, and respectively, to embed $B_n$, $\Sigma_n$ and $PwB_n$ inside $wB_n$. 
The canonical embeddings are given by the identifications $B_n= \langle \sigma_i \rangle_{i=1}^{n-1}\subseteq vB_n$, $
\Sigma_n= \langle \tau_i \rangle_{i=1}^{n-1}\subseteq vB_n$ and $PvB_n= \langle \sigma_{i,j}\rangle _{i\neq j} \subseteq vB_n$. 
From here on, when we refer to the restriction of a map on $vB_n$ (resp. $wB_n)$ to $B_n$, $\Sigma_n$ or $PvB_n$ (resp. $PwB_n$), we are referring to the canonical embeddings of these groups.

Recall from the introduction that a map $\phi$ is called cyclic (resp. abelian) if its image is cyclic (resp. abelian). 

\begin{lem}\label{Lem:TauCommute}
If $\phi \from vB_n \to G$ is a group homomorphism so that $\phi$ restricted to either $\Sigma_n$ or $B_n$ is abelian, then $\phi$ abelian.
\end{lem}

\begin{proof}
Suppose $\phi$ restricted to $\Sigma_n$ is abelian. 
The $\tau$ braid relation gives that $\phi(\tau_i)=\phi(\tau_{i+1})$ for all $i$, and so $\phi$ is cyclic on $\Sigma_n$. 
Denote $\phi(\tau_i)=g$. 
Applying $\phi$ to the mixed braid relation yields 
\begin{align*}
    \phi(\tau_{i+1}\sigma_i\tau_{i+1})&=\phi(\tau_i\sigma_{i+1}\tau_i)\\
    g\phi(\sigma_i)g&=g\phi(\sigma_{i+t})g\\
    \phi(\sigma_i)&=\phi(\sigma_{i+1})
\end{align*} 
This shows that $\phi$ restricted to $B_n$ is also cyclic, and therefore $\phi$ is abelian.

Assume that $\phi$ restricted to $B_n$ is abelian. 
A similar argument using the braid relations shows that $\phi$ is cyclic on $B_n$, and the mixed braid relation shows that $\phi$ is cyclic on $\Sigma_n$.
\end{proof}

\begin{cor}\label{Cor:wBnAbelian}
If $\phi \from wB_n \to G$ is a group homomorphism so that $\phi$ restricted to either $\Sigma_n$ or $B_n$ is abelian, then $\phi$ abelian.
\end{cor}

\begin{proof}
The relations used in the proof of Lemma \ref{Lem:TauCommute} also hold in $wB_n$, so the same proof can be used here. 
Alternatively, let $p \from vB_n \to wB_n$ be the quotient projection map, and apply Lemma \ref{Lem:TauCommute} to the map $\phi \circ p$.
\end{proof}

The following lemma is a key step to proving Theorems \ref{thm:wBn} and \ref{Thm:vBn}. 
To use totally symmetric sets to count the cardinality of the image of a homomorphism, the homomorphism needs to split a totally symmetric set. 
This lemma shows that, under the right conditions, when some subset of the totally symmetric sets $\{A_i\}_{i=1}^{n}$ do not split under a map $\phi \from wB_n \to G$, we can use the images of the $A_i$'s which do not split to create a new totally symmetric set in the image.

\begin{lem}[Hot Air Balloon Lemma]\label{Lem:HotAirBalloon}
Let $\{A_{i_1},\cdots A_{i_m}\}$ be a subset of $\{A_1,\cdots, A_n\}$, the totally symmetric sets in $wB_n$. Let $\phi \from wB_n \to G$ be a non-abelian group homomorphism. 
Suppose $\phi$ does not split $A_{i_j}$ for all $i_j$, and each $\phi(A_{i_j})^2$ is a distinct element in the image. 
Then the set $\{\phi(A_{i_j})^2\}_{j=1}^m$ is a totally symmetric set in $\phi(wB_n)$ of size $m$.
\end{lem}

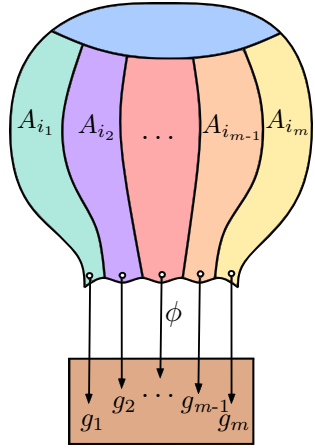
\begin{figure}[ht]
\includegraphics[width=115\unitlength]{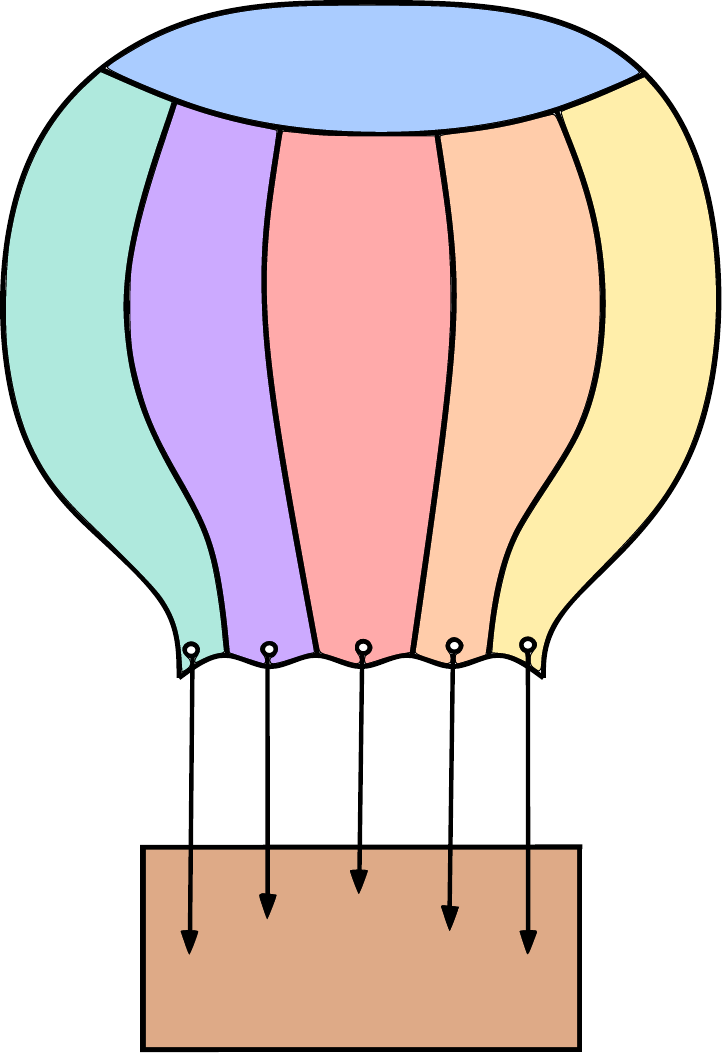}
\put(-112,120){\small{$A_{i_1}$}}
\put(-89,118){\small{$A_{i_{2}}$}}
\put(-42,118){\small{$A_{i_{m\text{-}1}}$}}
\put(-18,120){\small{$A_{i_m}$}}
\put(-88,7){\small{$g_1$}}
\put(-76,14){\small{$g_2$}}
\put(-50,14){\small{$g_{\text{\textit{m}-1}}$}}
\put(-36,7){\small{$g_m$}}
\put(-56,47){$\phi$}
\put(-65,18){$\ldots$}
\put(-65,116){$\ldots$}
\caption{Schematic diagram for the Hot Air Balloon Lemma.}
\end{figure}

\begin{proof}
We will prove this lemma for the case where $\{A_{i_1},\cdots A_{i_m}\}=\{A_1,\cdots, A_m\}$, as all other cases follow from an analogous proof with possible re-indexing. 

Let $g_i=\phi(A_i)$. We will show that the set $\{g_i^2\}_{i=1}^m$ is a totally symmetric set in $\phi(wB_n)$ of size $m$. 

By assumption, the $g_i^2$'s are distinct, so the set $\{g_i^2\}_{i=1}^m$ has $m$ elements. Notice that every element of $A_i$ is of the form $\sigma_{i,j}$, where the first index, $i$, remains fixed. Since $\phi$ does not split any of the totally symmetric sets $A_i$, $\phi(\sigma_{i,j})$ is determined by its first index $i$, i.e., $\phi(\sigma_{i,j})=g_i$.

For the commutation condition, applying $\phi$ to the braid relation in $wB_n$ shows
\begin{equation*}
    \begin{aligned}
    \sigma_{i,j}\sigma_{i,k}\sigma_{j,k} &= \sigma_{j,k}\sigma_{i,k}\sigma_{i,j}\\
    \phi(\sigma_{i,j})\phi(\sigma_{i,k})\phi(\sigma_{j,k}) &= \phi(\sigma_{j,k})\phi(\sigma_{i,k})\phi(\sigma_{i,j})\\
    g_i g_i g_j &= g_j g_i g_i,
    \end{aligned}
\end{equation*}
which shows that for each $i$ and $j$, $g_i^2$ and $g_j$ commute. In turn, this implies that $g_i^2$ and $g_j^2$ commute. 

To show the conjugation condition holds, notice that if $f g_i f^{-1}=g_j$ then $f g_i^2 f^{-1}=g_j^2$. 
Therefore, it suffices to show the conjugation condition holds for the set $\{g_i\}$. 
The following computations show that conjugation by $\phi(\tau_i)$ swaps $g_i$ and $g_{i+1}$ but fixes all other $g_k$.

First, we show that conjugation by $\phi(\tau_i)$ swaps $g_i$ and $g_{i+1}$. There are two cases to consider:
\begin{addmargin}[2em]{0em}

\noindent \textit{Case 1:} Suppose $i\leq n-2$. A similar computation described in Lemma \ref{lem:totSetsInwBn} shows that conjugation by $\tau_i$ swaps $\sigma_{i,i+2}$ with $\sigma_{i+1,i+2}$. Thus
\begin{equation*}
    \begin{aligned}
    \phi(\tau_i\sigma_{i,i+1}\tau_i) &= \phi(\sigma_{i+1,i+2})\\
    \phi(\tau_i)g_i\phi(\tau_i) &= g_{i+1}.
    \end{aligned}
\end{equation*}

\noindent \textit{Case 2:} Suppose $i>n-2$, which implies that $i=n-1$.
A similar computation to the one above shows that conjugation by $\tau_{i-1}$ swaps $\sigma_{i,i-1}$ and $\sigma_{i+1,i-1}$ and that conjugation by $\phi (\tau_i)$ swaps $g_i$ and $g_{i+1}$.
\end{addmargin}

\noindent Next, we show that for $g_k$, where $k\neq i, i+1$, that $g_k$  remains fixed under conjugation by $\phi(\tau_i)$. 
To prove this, we must consider six different cases on $k$. 
In each case, it suffices to find a single $\sigma_{k,{-}}$ with first index $k$ so that $\sigma_{k,{-}}$ is fixed under conjugation by $\tau_i$.

\begin{addmargin}[2em]{0em}

\noindent \textit{Case 1:} If $k>i+1$,  $k\neq n$, then $\sigma_{k,k+1}$ is fixed by the commutation relations in $wB_n$.

\noindent \textit{Case 2:} If $k=n$, $i= n-2$, then $\sigma_{k,k-3}$ is fixed by the following computation:
\begin{equation*}
    \begin{aligned}
    \tau_{n-2} \sigma_{n,n-3} \tau_{n-2} &= \tau_{n-2} \tau_{n-1} \tau_{n-2} \sigma_{n-3} \tau_{n-3} (\tau_{n-2} \tau_{n-1} \tau_{n-2})\\
    &= \tau_{n-1} \tau_{n-2} \tau_{n-1} \sigma_{n-3} \tau_{n-3} (\tau_{n-1}) \tau_{n-2} \tau_{n-1}&\text{$\tau$ braid relation}\\
    &= \tau_{n-1} \tau_{n-2} \sigma_{n-3} \tau_{n-3} \tau_{n-2} \tau_{n-1}&\text{commute $\tau_{n-1}$ left} \\
    &= \sigma_{n,n-3}.
    \end{aligned}
\end{equation*}

\noindent \textit{Case 3:} If $k=n$, $i\neq n-2$, then $\sigma_{k,k-1}$ is fixed by the commutation relations in $wB_n$.

\noindent \textit{Case 4:} If $k<i$, $k\neq 1$, then $\sigma_{k, k-1}$ is fixed by the commutation relations in $wB_n$.

\noindent \textit{Case 5:} If  $k=1$, $i\neq 2$, then $\sigma_{k,k+1}$ is fixed by the commutation relations in $wB_n$.

\noindent \textit{Case 6:} If $k=1$ and $i= 2$, then $\sigma_{k, k+3}$ is fixed by a similar computation as in Case 2.
\end{addmargin}

Thus, for every $k\neq i, i+1$, there exists $\sigma_{k,j}$ that is fixed under conjugation by $\tau_i$. This shows that 
\begin{equation*}
    g_k = \phi(\sigma_{k,-}) = \phi(\tau_i\sigma_{k,-}\tau_i) =  \phi(\tau_i)g_k\phi(\tau_i),
\end{equation*}
which proves the conjugation condition in the definition of a totally symmetric set holds, and we have proven our claim.
\end{proof}

\begin{rmk}
The Hot Air Balloon Lemma is stated for $wB_n$, however it is also true for $vB_n$. 
In $vB_n$ the sets $A_i$ are not totally symmetric, but they do satisfy the conjugation condition, which is the only condition needed in the proof. 
\end{rmk}

\section{Finite Image homomorphisms  of the Virtual and Welded Braid Groups}\label{Sec:ApplicationtoWBn}

In this Section, we prove the classification theorems on the size of finite images of homomorphisms of both $wB_n$ and $vB_n$.

\subsection{Proof of Theorem \ref{thm:wBn}}

First we prove the classification theorem for the welded braid group, $wB_n$. 
We hope that this is a first step in classifying non-cyclic homomorphisms $wB_n \to G$, where $G$ is a finite group.\\

\noindent \textbf{Theorem \ref{thm:wBn}:}
 \textit{Let  $n> 5$, and let $\phi:wB_n\to G$ be a group homomorphism to a finite group, $G$.
One of the following must be true:
\begin{enumerate}
\item $\phi$ is abelian.
\item $\phi$ restricted to $PwB_n$ is cyclic.
\item $\abs{\phi(wB_n)}\geq 2^{n-2}(n-1)!$
\item For all $i$ and $j$, $\phi$ maps each $A_i$ to a single element with $\phi(A_i)^2\neq\phi(A_j)^2$, and 
\[\abs{\phi(wB_n)} \geq  \left( \left\lfloor\frac{n}{2}\right\rfloor +1\right)(3^{\lfloor\frac{n}{2}\rfloor - 1})\left\lceil \frac{n}{2} \right\rceil!.\]
\end{enumerate}
}
\vspace{5mm}

In the statement of Theorem \ref{thm:wBn}, the requirement that $n>5$ is only necessary for Part (3) due to the applications of Lemma \ref{Cor:NoncyclicSplitsBoth} and Theorem \ref{thm:extendedbound}. All other conditions hold for $n\geq 4$.

\begin{figure}[ht]
\begin{picture}(390,125)
\includegraphics[width=400\unitlength]{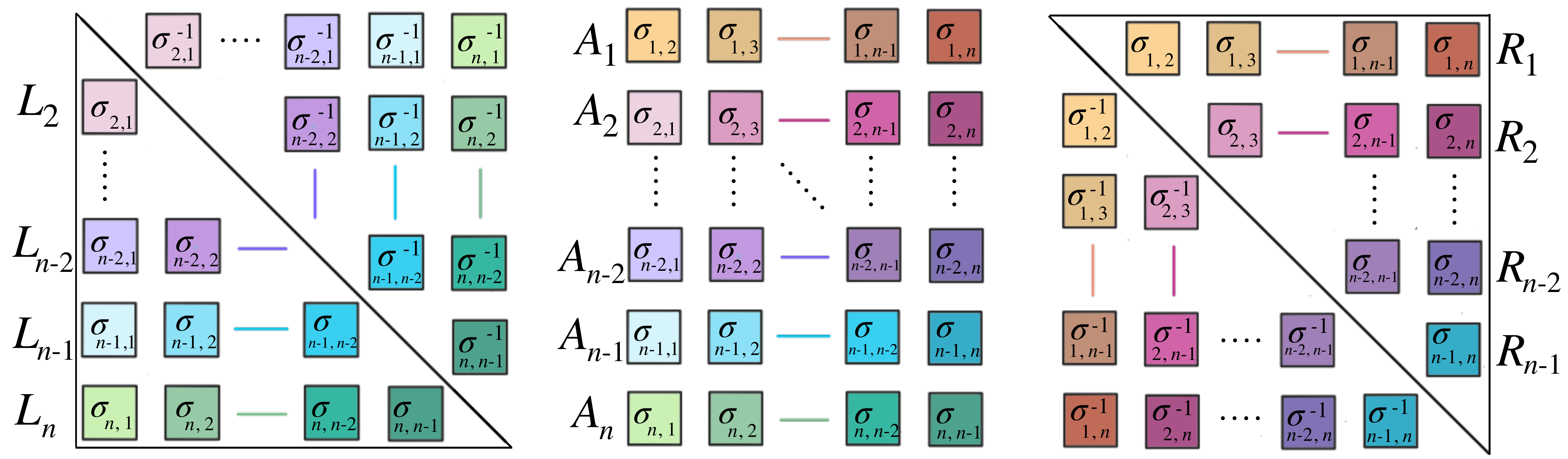}
\put(-400,-10){(a)}
\put(-260,-10){(b)}
\put(-145,-10){(c)}
\end{picture}
\caption{(a) Left diagram.  (b) Full diagram.  (c) Right diagram.}
\label{Fig:grids} 
\end{figure}

The proof of Theorem \ref{thm:wBn} is inspired by Figure \ref{Fig:grids}. 
We consider cases on whether $\phi$ splits various rows and columns of the diagrams. 
The rows of the Full diagram, as seen in Figure \ref{Fig:grids} (b), are the totally symmetric sets, $A_i$, from Lemma \ref{lem:totSetsInwBn}. 
In the Left diagram, the rows of the outlined triangle are the totally symmetric sets, $L_i$. 
The rows above the outlined triangle are the inverses of the columns within the outlined triangle. 
Both the columns in the outlined triangle and the rows above the outline are not totally symmetric sets, but satisfy the conjugation condition. 
This can be verified by similar computations described in Lemma \ref{lem:totSetsInwBn}. 
The Right  diagram has an analogous form.  
The rows of the outlined triangle are the totally symmetric sets, $R_i$. 
The columns below the outlined triangle are the inverses of the rows within the triangle, and both satisfy the conjugation condition. 

\begin{proof}[Proof of Theorem \ref{thm:wBn}]
Let us suppose that $\phi$ is non-abelian and that $\phi$ restricted to $PwB_n$ is non-cyclic. 
We consider cases on whether or not $\phi$ splits the totally symmetric sets $A_i$.

Case 1:  
Suppose that there exists an $i$ so that $\phi$ splits $A_i$. 
Since $A_i$ is a totally symmetric set with size $n-1$, applying Proposition \ref{Prop:SizeOfImage} yields 
\[
\abs{\phi(wB_n)}\geq 2^{n-2}(n-1)!.
\]

Case 2: Suppose $\phi$ does not split any of the $A_i$'s.
Denote $\phi(A_i)=\{g_i\}$. 
Further suppose that there exists $i_0$ and $j_0$ so that $g_{i_0}^2\neq g_{j_0}^2$. 
Notice this implies $g_{i_0}\neq g_{j_0}$. 
Since $\phi$ is non-abelian, we may assume by Lemma \ref{Lem:TauCommute} that $\phi$ is non-cyclic on $\Sigma_n$ and that $\phi(\tau_i)\neq id$. 
We consider cases on $i_0$ and $j_0$ with the goal to apply the Hot Air Balloon Lemma.

\begin{addmargin}[2em]{0em}
\noindent \textit{Subcase 1:} Suppose $i_0,j_0<n$. We will use the Right  diagram in Figure \ref{Fig:grids} to conclude that $g_1,\cdots, g_{n-1}$ are distinct. 
By assumption, $g_{i_0}\neq g_{j_0}$ which implies that $\phi(\sigma_{{i_0},n})\neq \phi(\sigma_{{j_0},n})$, and therefore $\phi(\sigma_{{i_0},n}^{-1})\neq \phi(\sigma_{{j_0},n}^{-1})$. 
The bottom row of the Right  diagram contains both $\sigma_{{i_0},n}^{-1}$ and  $\sigma_{{j_0},n}^{-1}$.  
Even though the bottom row of the Right  diagram is not a totally symmetric set, it does satisfy the conjugation condition.
Since $\phi(\sigma_{i_0,n}^{-1}) \neq \phi(\sigma_{i_1,n}^{-1})$, Remark \ref{rmk:OnlyConjugate} implies that $\phi$ splits the bottom row.
Thus $\phi(\sigma_{i,n}^{-1})\neq \phi(\sigma_{j,n}^{-1})$ for all $i,j< n$, which shows that $g_i\neq g_j$, for all $i,j<n$.  
Since $g_{i_0}^2\neq g_{j_0}^2$ by assumption, Remark \ref{remk:AfterEqualatSamePower} shows the each of the $g_i^2$ are unique. 
Thus, we have shown all of the hypotheses of the Hot Air Balloon Lemma are satisfied, and $\{g_1^2,\cdots, g_{n-1}^2\}$ is a totally symmetric set in the image of $\phi$ of size $n-1$.
Proposition \ref{Prop:SizeOfImage} yields 
\[
\abs{\phi(wB_n)}\geq 2^{n-2}(n-1)!.
\]

\noindent \textit{Subcase 2:} Suppose $i_0,j_0>1$. 
An analogous argument to Subcase 1 using the Left  diagram from Figure \ref{Fig:grids} concludes that $\{g_2^2,\cdots, g_n^2\}$ is a totally symmetric set in the image of $\phi$ of size $n-1$.
Proposition \ref{Prop:SizeOfImage} yields 
\[
\abs{\phi(wB_n)}\geq 2^{n-2}(n-1)!.
\]

\noindent \textit{Subcase 3:} Suppose $i_0=1$ and $j_0=n$, which implies that $g_1 \neq g_n$, and further that $\phi(\sigma_{1,-}) \neq \phi(\sigma_{n,-})$.
Looking at the Full diagram in Figure \ref{Fig:grids} (b), Subcase 3 analyzes when the top and bottom rows of the Full diagram are mapped to different elements. 

We now analyze where $\phi$ can send the second row.

Suppose first that $\phi$ maps $A_2$, or all the elements of the second row, to $g_1$.
Then in Figure \ref{Fig:grids} (a), the Left  diagram, we notice that $L_2$ and $L_{n}$ map to different elements.
Therefore, two elements in the top row of Figure \ref{Fig:grids} (a) map to different elements.
Since the top row satisfies the conjugation condition, we have that the top row must split.
Since $g_{i_0}^2\neq g_{j_0}^2$ by assumption, Remark \ref{remk:AfterEqualatSamePower} shows that each $g_i^2$ is unique.
Therefore, by the Hot Air Balloon Lemma, the set $\{ g_2^2, g_3^2, \ldots, g_n^2 \}$ is a totally symmetric set of size $n-1$.
Proposition \ref{Prop:SizeOfImage} yields 
\[
\abs{\phi(wB_n)}\geq 2^{n-2}(n-1)!.
\]

A similar argument follows for when $\phi$ maps $A_2$, or all the elements of the second row, to $g_n$, but this time we consider Figure \ref{Fig:grids} (c), the Right  diagram.
Since $R_1$ and $R_2$ map to different elements, the bottom row of the Right  diagram must split as it satisfies the conjugation relation.
Since $g_{i_0}^2\neq g_{j_0}^2$ by assumption, Remark \ref{remk:AfterEqualatSamePower} shows that each $g_i^2$ is unique.
Therefore, by the Hot Air Balloon Lemma, the set $\{ g_1^2, g_2^2, \ldots, g_{n-1}^2 \}$ is a totally symmetric set of size $n-1$.
In this case, Proposition \ref{Prop:SizeOfImage} will again yield 
\[
\abs{\phi(wB_n)}\geq 2^{n-2}(n-1)!.
\]

Finally, suppose that $\phi$ sends $A_2$ to an element $g_2$ where $g_2 \neq g_1, g_n$.
Then in Figure \ref{Fig:grids} (a), the Left  diagram, we notice that $L_2$ and $L_{n}$ map to different elements.
Therefore, two elements in the top row of Figure \ref{Fig:grids} (a) map to different elements.
Since the top row satisfies the conjugation condition, we have that the top row must split.
Similarly, in Figure \ref{Fig:grids} (c), the Right  diagram, we notice that $R_1$ and $R_{2}$ map to different elements.
Therefore, two elements in the bottom row of Figure \ref{Fig:grids} (c) map to different elements, and since the bottom row satisfies the conjugation condition the bottom row must split.
Notice that we must have that $\phi$ sends each $A_i$ to a unique element.
Indeed, suppose that $g_i = g_j$ for some $i, j$.
This implies that either the top row of the Left  diagram or the bottom row of the Right  diagram cannot split since these rows have the conjugation relation, which is a contradiction to the above.
Since $g_{i_0}^2\neq g_{j_0}^2$ by assumption, Remark \ref{remk:AfterEqualatSamePower} shows the each $g_i^2$ is unique.
By the Hot Air Balloon Lemma, the set $\{ g_1^2, g_2^2, \ldots, g_n^2 \}$ is a totally symmetric set of size $n$.
In this case, Proposition \ref{Prop:SizeOfImage} will yield that
\[
\abs{\phi(wB_n)}\geq 2^{n-1}(n)!.
\]
\end{addmargin}

Case 3: Suppose $\phi$ does not split any of the $A_i$'s, and  $\phi(A_i)^2=\phi(A_j)^2$ for all $i$ and $j$. 
Notice that the Hot Air Balloon Lemma does not apply, and that none of the $A_i$'s are split by $\phi$.
In this case, we use the fact that $\phi$ restricted to $B_n$ is non-cyclic. 
Applying Theorem \ref{thm:extendedbound} to $\phi$ restricted to $B_n$, we get 
\[\abs{\phi(wB_n)} \geq \abs{\phi(B_n)} \geq \left( \left\lfloor\frac{n}{2}\right\rfloor +1\right)(3^{\lfloor\frac{n}{2}\rfloor - 1})\left\lceil \frac{n}{2} \right\rceil!.\]

\end{proof}

\subsection{Proof of Theorem \ref{Thm:vBn}}

In this section, we provide a proof for Theorem \ref{Thm:vBn} which gives a lower bound on the size of $vB_n$'s smallest non-cyclic finite quotient.
By considering whether or not a homomorphism factors through $wB_n$, we may apply our classification of homomorphisms from $wB_n \to G$, or the necessary condition for the existence of a homomorphism $B_n \to G$, to determine a classification of the size of finite images of homomorphisms from $vB_n$.\\

\noindent \textbf{Theorem \ref{Thm:vBn}:}
\textit{
Let $n>5$, and let $\phi \from vB_n \to G$ be a group homomorphism to a finite group, $G$. One of the following must be true:
\begin{enumerate}
    \item $\phi$ is abelian.
    \item $\phi$ factors through $wB_n$, and either
    \begin{enumerate}
        \item  $\phi$ restricted to $PwB_n$ is cyclic.
        \item $\abs{\phi(vB_n)}\geq 2^{n-2}(n-1)!$
\item For all $i$ and $j$, $\phi$ does not split $A_i$, $\phi(A_i)^2\neq\phi(A_j)^2$, and 
\[\abs{\phi(vB_n)} \geq  \left( \left\lfloor\frac{n}{2}\right\rfloor +1\right)(3^{\lfloor\frac{n}{2}\rfloor - 1})\left\lceil \frac{n}{2} \right\rceil!.\]
    \end{enumerate}
    \item $\phi$ does not factor through $wB_n$ and 
\[\abs{\phi(vB_n)} \geq  \left( \left\lfloor\frac{n}{2}\right\rfloor +1\right)(3^{\lfloor\frac{n}{2}\rfloor - 1})\left\lceil \frac{n}{2} \right\rceil!.\]
\end{enumerate}
}

\begin{proof}[Proof of Theorem \ref{Thm:vBn}]
Suppose $\phi$ is not abelian. If $\phi$ factors through $wB_n$, then by Theorem \ref{thm:wBn}, one of either (2)(a), (2)(b), or (2)(c) must be true.
If $\phi$ does not factor through $wB_n$ and $\phi$ is non-abelian, Lemma \ref{Lem:TauCommute} gives that $\phi$ restricted to $B_n$ is non-abelian, and hence non-cyclic. 
Applying Theorem \ref{thm:extendedbound} to $\phi$ restricted to $B_n$ gives that (3) must be true.
\end{proof}

\bibliographystyle{plain}
\bibliography{virtualBraidBibli}

\end{document}